\documentclass[11pt]{amsart}
\usepackage{amsmath,amssymb,enumerate,etoolbox,mathrsfs,xcolor,url,hyperref}
\usepackage[normalem]{ulem}
\usepackage[all]{xy}
\usepackage{tikz}

\usepackage[top=30mm,right=30mm,bottom=30mm,left=20mm]{geometry}

\newtheorem{theorem}{Theorem}[section]
\newtheorem{thm}[theorem]{Theorem}

\newtheorem{theor}{Theorem}

\newtheorem{prop}[theorem]{Proposition}
\newtheorem*{proposition*}{Proposition}
\newtheorem{lem}[theorem]{Lemma}

\newcommand{\cO}{\mathcal{O}}

\newcommand{\AAA}{\mathsf{A}}
\newcommand{\SSS}{\mathsf{S}}

\renewcommand{\setminus}{\smallsetminus}

\newcommand{\Perm}{\mathrm{Sym}}
\newcommand{\len}{\mathrm{len}}
\newcommand{\fix}{\mathrm{fix}}
\newcommand\Interval{J}
\newcommand{\cn}{\mathsf{cn}}

\numberwithin{equation}{section}

\begin{document}

\title{Squares of conjugacy classes in alternating groups}

\author{Michael Larsen}
\email{mjlarsen@indiana.edu}
\address{Department of Mathematics\\
    Indiana University \\
    Bloomington, IN 47405\\
    U.S.A.}

\author{Pham Huu Tiep}
\email{pht19@math.rutgers.edu}
\address{Department of Mathematics\\
    Rutgers University \\
    Piscataway, NJ 08854-8019 \\
    U.S.A.}
\thanks{The first author was partially supported by the NSF 
grant DMS-2001349 and Simons Foundation Fellowship 917214. The second author gratefully acknowledges the support of the NSF (grants DMS-1840702 and DMS-2200850), Simons Foundation Fellowship 918093, and the Joshua Barlaz Chair in Mathematics.}

\begin{abstract}
We extend to alternating groups $\AAA_n$ several results about symmetric groups asserting that under various conditions on a conjugacy class, or more generally, a normal subset, $C$ of $\SSS_n$, we have $C^2 \supseteq \AAA_n\setminus\{1\}$.
\end{abstract}

\maketitle

\section{Introduction}
Let $\SSS_n$ denote a symmetric group and $\AAA_n$ its subgroup of even permutations.  Let $C$ be the conjugacy class of $n$-cycles in $\SSS_n$.
Andrew Gleason \cite{Husemoller} seems to have been the first to observe that $C^2 = \AAA_n$. This result has been generalized in various ways (see, for instance, \cite{Bertram,Brenner,BR,Hsu,LS1}).
We now know that the condition $C^2 = \AAA_n$ is, by any reasonable definition, generic behavior for a conjugacy class of $\SSS_n$.

In this note, we focus on the conjugacy classes in alternating groups.  Most conjugacy classes in $\AAA_n$ are conjugacy classes 
(of even permutations) in $\SSS_n$; the 
novelty lies in understanding what happens with $\AAA_n$-classes associated to partitions of $n$ consisting of distinct odd parts.
A key tool is Frobenius' formula counting the number of triples $(x_1,x_2,x_3)$ with product $1$, where $x_i$ belongs to a given conjugacy class $C_i$ for each $i$.
To use this method, we need good upper bounds on irreducible character values for $\AAA_n$.
There is an extensive literature of such bounds for symmetric groups (see, e.g., \cite{Biane, FL, LS2, MSP, RS, Roichman} and the references contained therein), so what is needed 
is upper bounds specifically for those irreducible characters of $\AAA_n$ which are not restrictions of characters of $\SSS_n$.

Our main results are as follows.  Refining the result of Gleason  \cite{Husemoller}, we have

\begin{theor}
\label{Gleason}
Let $C$ and $D$ be (possibly equal) conjugacy classes of $\AAA_n$, $2 \nmid n\ge 7$, each consisting of $n$-cycles.  Then $CD \supseteq \AAA_n\setminus\{1\}$.
\end{theor}

Our supplement to the main theorem of \cite{LS2} is as follows:

\begin{theor}
\label{char-bound}
For all $\epsilon > 0$ there exists $N=N(\epsilon) >0$ such that if $n>N$ and $x\in \AAA_n$ is an element for which $x^{\AAA_n}\neq x^{\SSS_n}$, then $|\chi(x)|\le \chi(1)^\epsilon$
for every irreducible character $\chi$ of $\AAA_n$.
\end{theor}

From this, we deduce the following three consequences.

\begin{theor}
\label{Covering}
For all sufficiently large $n$, if $x,y\in \AAA_n$ are elements which are conjugate in $\SSS_n$ but such that $x^{\AAA_n}\neq x^{\SSS_n}$, then $x^{\AAA_n} y^{\AAA_n} \supseteq \AAA_n\setminus\{1\}$.
\end{theor}

\begin{theor}
\label{Normal Subset}
For every $0<\alpha<1/4$, there exists $N = N(\alpha) > 0$ such that, if $n\ge N$ and
$W\subseteq \AAA_n$ is a normal subset satisfying
$|W| \ge e^{-n^\alpha} |\AAA_n|$, then $W^2\supseteq \AAA_n\setminus\{1\}$.
  \end{theor}

Theorem \ref{Normal Subset} is similar in spirit to \cite[Theorem~6.4]{LST} but gives a greatly improved bound.  In the preprint \cite{LST0}, we claimed this stronger bound, but we did not properly justify it.
Here we give a corrected version of that argument.

The next consequence is concerned with the {\it covering number} $\cn(C)=\cn(C,G)$ which is the smallest positive integer $k$ such that $C^k=G$ for a 
given conjugacy class $C$ in a finite group $G$.

\begin{theor}
\label{an-cn}
Let $C=g^{\AAA_n}$ be the $\AAA_n$-conjugacy class of any element $g \in \AAA_n$ with $g^{\AAA_n} \neq g^{\SSS_n}$. Then the following 
statements hold.

\begin{enumerate}[\rm(i)]
\item Suppose $2 \nmid n \geq 5$ and $g$ is an $n$-cycle. Then $\cn(C) = 2$ if $n \equiv 1 \pmod{4}$ and $n \geq 7$, and 
$\cn(C)=3$ otherwise. 
\item Suppose $n$ is sufficiently large, and let $\kappa(g)$ denote the number of disjoint cycles of $g$ of length congruent to $3 \pmod{4}$. Then
$\cn(C) = 2$ if $2|\kappa(g)$, and $\cn(C) = 3$ if $2 \nmid \kappa(g)$.
\end{enumerate}
\end{theor}

\section{Classes of $n$-cycles}

Let $n$ be a positive integer and $\lambda\vdash n$ a partition. We say an element $g\in \AAA_n$ is \emph{covered by $\lambda$} if whenever $C$ and $D$ are $\AAA_n$-conjugacy classes whose cycle structure in $\SSS_n$ is given by $\lambda$, then $g\in CD$.

In addition to the notation $\lambda = (\lambda_1 \geq \lambda_{2} \geq \cdots \geq \lambda_k)$, we also use the notation $1^{m_1}2^{m_2} \ldots n^{m_n}$ to denote the partition with $m_i$ parts equal to $i$ for $1 \leq i \leq n$;
in particular, $n^1$ is the one-part partition of $n$.
If $\lambda_k$ is the smallest (positive) part of $\lambda= (\lambda_1 \geq \lambda_{2} \geq \cdots \geq \lambda_k)$, we denote by $\bar{\lambda}$ the partition of $n-\lambda_k$ with parts $\lambda_1,\ldots,\lambda_{k-1}$.

\begin{lem}
\label{Cancel}
With notation as above, if every non-trivial element of $\AAA_{n-\lambda_k}$ is covered by $\bar{\lambda}$, then every non-trivial element of $\AAA_n$ with at least $\lambda_k$ fixed points is covered by $\lambda$.
\end{lem}

\begin{proof}
Without loss of generality, we may assume that each of $n,n-1,\ldots,n-\lambda_k+1$ is fixed by $g$.  If $C$ and $D$ are conjugacy classes of $\AAA_n$ associated to the partition $\lambda$,
then there exist conjugacy classes $C'$ and $D'$ of $\AAA_{n-\lambda_k}$ such that for every $c'\in C'$,
$$c := c'z \in C \mbox{ for }z:= (n,n-1,\cdots,n-\lambda_k+1)$$
and for every $d'\in D'$,
$$d := d'\cdot z^{-1}\in D.$$
By hypothesis, $g$ (viewed as an element in $\AAA_{n-k}$) belongs to $C'D'$ and thus
$g=c'd'$ for some $c' \in C$ and some $d' \in D$. Now $g = cd \in CD$.
\end{proof}

\begin{lem}
\label{partial cancellation cycle}
Suppose that $n-r$ is odd and a non-trivial element $h$ of $\AAA_{n-r}$ is covered by the partition $(n-r)^1$. Then 
$h$ viewed as an element $g$ of $\AAA_n$ that fixes each of $n,n-1,\ldots,n-r+1$ is covered by the partition $n^1$.
In particular, if $n-r$ is odd and every non-trivial element of $\AAA_{n-r}$ is covered by $(n-r)^1$, then 
every nontrivial element of $\AAA_n$ with at least $r$ fixed points is covered by $n^1$.
\end{lem}

\begin{proof}
Let $C'$ and $D'$ be conjugacy classes in $\AAA_{n-r}$ associated to the partition $(n-r)$. By assumption,
$h=c'd'$ with $c' \in C'$ and $d' \in D'$. We can write $c' \in C'$ of the form 
$$c' = (c_1,\ldots,c_{n-r-1},n-r).$$
To any such $c'$, we associate an element
$$c := (c_1,\ldots,c_{n-r-1},n-r,\ldots,n-1,n) \in \AAA_n$$
whose class in $\SSS_n$ is given by the partition $n^1$.  As $n-r$ is odd, the conjugacy class of $c$ in $\AAA_n$ depends only on the conjugacy class $C'$ of $c'$ in $\AAA_{n-r}$.

Likewise, we can write $d' \in D'$ as
$$d' = (n-r, d_1,\ldots,d_{n-r-1}),$$
and to such $d'$, we associate
$$d := (n,n-1,\ldots,n-r,d_1,\ldots,d_{n-r-1}).$$
Now observe that $c'd' = h$ in $\AAA_{n-r}$ implies 
that $cd = g$ in $\AAA_n$.
\end{proof}

We denote by $\chi = \chi_\lambda$ the irreducible character of $\SSS_n$ associated to the partition $\lambda\vdash n$.
We say the Young diagram of $\lambda$ is a \emph{hook diagram} and $\chi_\lambda$ is a \emph{hook character} if $\lambda_2\le 1$, or, equivalently, every box in the Young diagram of $\lambda$ is in the first row
or the first column.   By the \emph{size} of a hook character, we mean the shorter of the first row and the first column of its diagram.
By the Murnaghan-Nakayama rule, if $\chi$ is any irreducible character of $\SSS_n$
and $g$ is an $n$-cycle, then $\chi(g) \in \{\pm 1\}$ if $\chi$ is a hook character, and $\chi(g) = 0$ otherwise.

\begin{lem}
\label{hooks}
Let $g\in \SSS_n$ be an element with at most one fixed point.  Let $\chi$ be a hook character of $\SSS_n$ of size $k$.  Then
$$|\chi(g)| \le \sum_{0\le i < k/2}\binom{\lceil n/2\rceil-1}i.$$
\end{lem}

\begin{proof}
As $g$ has at most one fixed point, if $m$ is the number of parts of its associated partition 
$\mu = (\mu_1 \geq \mu_2 \geq \cdots \geq \mu_m)$ of $n$, then $m\le \lceil  n/2\rceil$,
and $\mu_1, \ldots,\mu_{m-1} \geq 2$. 
We apply the Murnaghan-Nakayama rule,
noting that for each $r \geq 1$ and each hook diagram with more than $r$ boxes, 
there are at most two rim hooks of length $r$ which can be removed from the diagram,
one from the first row and one from the first column.  Removing such a rim hook shortens
the first row or column by $r$.  Removing $m-1$ rim hooks, each of length $\mu_j \ge 2$, $1 \leq j \leq m-1$, 
what remains is a hook diagram.  Without loss of generality we may assume that $\lambda$ has $k$ rows. Then the
number $i$ of rim hooks removed from the first column is at most $(k-1)/2$, and the subset of the first $m-1$ rim hooks which are removed from the first column can be chosen in at most 
$\binom{\lceil n/2\rceil-1}i$ ways (by specifying the lowest node of each of these $i$ rim hooks).
\end{proof}

\begin{prop}
\label{almost derangement}
For $n\ge 5$ odd, every element of $\AAA_n$ with at most one fixed point is covered by the partition $n^1$, with the exception of elements with one fixed point in $\AAA_5$.
\end{prop}

\begin{proof}
Let $x$ and $y$ denote any $n$-cycles in $\AAA_n$ and $z$ an arbitrary element with at most one fixed point.
By Frobenius' formula, it suffices to prove that
\begin{equation}
\label{Frob}
\sum_\chi \frac{\chi(x)\chi(y)\chi(z)}{\chi(1)} \neq 0,
\end{equation}
where the sum is taken over all irreducible characters $\chi$ of $\AAA_n$.  
For $5\le n\le 11$ odd, we use {\sf GAP} \cite{GAP} to examine all non-trivial conjugacy classes in $\AAA_n$, and except when $n=5$ and $z$ is the product of two disjoint transpositions, the sum \eqref{Frob} is non-zero.

\smallskip
Assume now that $n \geq 13$. In \eqref{Frob}, 
we may sum just over characters such that $\chi(x)\neq 0$.  Any such character which extends to a character of $\SSS_n$ must be a hook character.
By \cite[Theorem~2.5.13]{JK}, any such character not extending to $\SSS_n$ must be one of the two irreducible factors $\varphi_1,\varphi_2$ of the restriction to $\AAA_n$ of the unique hook character of size $(n+1)/2$.
Moreover, $|\varphi_i(x)|$ and $|\varphi_i(y)|$ are bounded above by $\sqrt n+1$.  If $z$ is an $n$-cycle, then $|\varphi_i(z)| \le \sqrt n+1$.  Otherwise, Lemma~\ref{hooks} and  \cite[Theorem~2.5.13]{JK} give the upper bound
$$|\varphi_i(z)| \le \frac 12\sum_{0\le i<n/4}\binom{(n-1)/2}i.$$

We break up the sum \eqref{Frob} into the trivial character, hook characters of size $\ge 2$, and the $\varphi_i$.  The sum of the absolute value of $|\chi(x)\chi(y)\chi(z)/\chi(1)|$ as $\chi$ ranges over hook characters of size $\ge 2$ is bounded above by
$$\Sigma=:\sum_{k=2}^{(n-1)/2} \sum_{0\le i<k/2} \frac{\binom{(n-1)/2}{i}}{\binom{n-1}{k-1}}.$$

Note that in the inner sum $\Sigma_k$ for each $k$, the summands increase monotonically with $i$. Moreover, the leading term
of $\Sigma_k$ (with $i = \lfloor (k-1)/2 \rfloor$) decreases when we change from $k$ to $k+2$.
This leading term for $k=6$ is $15/(n-2)(n-4)(n-5)$, and for $k=7$ is $15/(n-2)(n-4)(n-6)$.

Therefore, the double sum $\Sigma$ can be expressed as
$$\frac 1{n-1}+\Bigl(\frac 2{(n-1)(n-2)}+\frac 1{n-2}\Bigr)+\frac 3{(n-2)(n-3)}+\frac 3{(n-2)(n-4)}$$
(accounting for $\Sigma_2$, $\Sigma_3$, and the leading terms of $\Sigma_4$ and $\Sigma_5$)
plus a sum of the remaining term of $\Sigma_4$, the two remaining terms of $\Sigma_5$, and all terms of $\Sigma_k$ for $k\ge 6$,
for a total of at most $(n+1)^2/16-6$ terms, 
each bounded above by
$$\frac{15}{(n-2)(n-4)(n-6)}.$$
In total, we have an upper bound
\begin{align*}
\frac 1{n-1}+\Bigl(\frac 2{(n-1)(n-2)}+\frac 1{n-2}\Bigr)+\frac 3{(n-2)(n-3)} &+\frac 3{(n-2)(n-4)} \\
&+ \frac{15((n+1)^2/16-6)}{(n-2)(n-4)(n-6)} < \frac 12
\end{align*}
for $n\ge 13$.

The expression 
$$\frac{\bigl(\frac{\sqrt n+1}2\bigr)^3}{\binom{n-1}{(n-1)/2}} =\bigl(\frac{\sqrt n+1}2\bigr)^3\prod_{i=1}^{(n-1)/2}\frac i{4i-2}$$
is less than $1/4$ for $n=13$ and decreases monotonically as $n$ ranges through odd integers $\ge 13$.  Likewise, 
$$\frac{\bigl(\frac{\sqrt n+1}2\bigr)^2\sum_{0\le i<(n-1)/4}\binom{(n-1)/2}i}{\binom{n-1}{(n-1)/2}} \le \frac{\bigl(\frac{\sqrt n+1}2\bigr)^2}2 \prod_{i=1}^{(n-1)/2}\frac i{2i-1}$$
is less than $1/4$ for $n=13$ and decreases as $n$ increases.
By the triangle inequality, the sum \eqref{Frob} cannot be zero.
\end{proof}

We now prove Theorem~\ref{Gleason}.

\begin{proof}
Let $g$ be any nontrivial element in $\AAA_n$ with exactly $k$ fixed points. Then in fact we have $k \leq n-3$.
If $k\le n-6$, we set $r := 2\lfloor k/2\rfloor$.  Thus, $n-r$ is odd and greater than $5$, so we can apply Proposition~\ref{almost derangement} and deduce that regarded as an element
of $\AAA_{n-r}$, $g$ is covered by $(n-r)^1$.  By Lemma~\ref{partial cancellation cycle}, $g$ is covered by $n^1$.

This leaves three cases.  If $k=n-4$ or $k=n-5$, then we set $r := n-7$.  We observed in the proof of Proposition~\ref{almost derangement} that every non-trivial element of $\AAA_7$ is covered by $7^1$.
If $k=n-3$, then we set $r := n-5$.  Again, we observed in the proof of Proposition~\ref{almost derangement} that $3$-cycles in $\AAA_5$ are covered by $5^1$. In all cases, we can finish by applying  Lemma~\ref{partial cancellation cycle}.
\end{proof}

\section{Explicit products}

The following result is a refinement of \cite[Proposition~6.1]{LS2}.
\begin{thm}
\label{Construction}
Let $\lambda=(\lambda_1, \ldots,\lambda_k) \vdash n$ be a partition with $k$ parts.  Then $\lambda$ covers every element of $\AAA_n$ with at least $8k+9$ fixed points.
\end{thm}

The rough idea behind the proof is the following.  Suppose we are trying to fill large segments (of lengths $\lambda_1,\ldots,\lambda_k$) with small segments of integer lengths, including many segments of length $1$.  We can so using the greedy algorithm, that is by putting in the largest segments that will fit first.  We can use this and Theorem~\ref{Gleason}
to show that target classes consisting of many short orbits and enough fixed points can be covered by $\lambda$.  
(To be more precise, since only even permutations are covered by the partitions $\lambda_i^1$, we must take care to fill the segments of even length two at a time.)

Given a target element $g \in \AAA_n$ with enough fixed points, we use the shape of $g$ to construct a suitable product $\gamma\delta$,
with $\gamma\in C$ and $\delta \in D$. Then we use Lemma \ref{Rebuild long cycles} (below) to conjugate $\delta$ by a transposition to slightly modify the shape of the resulting product $\gamma\delta'$, in particular, replacing 
an $m$-cycle and two fixed points of $\gamma\delta$ by a single $(m+2)$-cycle.
Making a sequence of modifications of this form, we can reach the shape of $g$ and thus show that 
every element with enough fixed points is covered by $\lambda$.

To implement this plan, we introduce some terminology and notation.
The \emph{interval} $[a,b]$, for integers $a\le b$, means the set $\{a,a+1,\ldots,b\}$; it has \emph{length} $\len [a,b] = 1+b-a$.
A (not necessarily full) \emph{packing} 
$$I_1 = [a_1,b_1],\ldots,I_r=[a_r,b_r]$$
of $[a,b]$ is a sequence of subintervals such that
$$b_r = a_{r-1}-1, b_{r-1} = a_{r-2}-1,\ldots,b_2 = a_1-1, b_1 = b.$$
In particular, the union $I_1\cup\cdots\cup I_r$ is the interval $[a_r,b]$, and the complement of this set, which we call the \emph{free space}, is the interval $[a,a_r-1]$
(unless $a=a_r$, in which case there is no free space and the packing is full).  

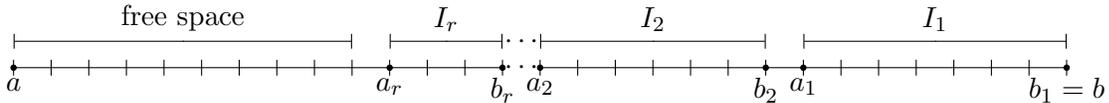
\begin{center}
\begin{tikzpicture}
  \draw[-] (-7.5,0)--(-1,0) node[right] {\hskip -3pt $\cdots$};
   \draw[-] (-.5,0)--(6.5,0);
  \draw[|-|][shift={(0,10pt)}] (3,0)--(6.5,0);
  \draw[|-|][shift={(0,10pt)}] (-.5,0)--(2.5,0);
  \draw[|-|][shift={(0,10pt)}] (-2.5,0)--(-1,0) node[right] {\hskip -3pt $\cdots$};
  \draw[|-|][shift={(0,10pt)}] (-7.5,0)--(-3,0);
  \draw[shift={(0,10pt)}] (4.75,0) -- (4.75,0) node[above] {$I_1$};
  \draw[shift={(0,10pt)}] (1,0) -- (1,0) node[above] {$I_2$};
  \draw[shift={(0,10pt)}] (-1.75,0) -- (-1.75,0) node[above] {$I_r$};
   \draw[shift={(0,10pt)}] (-5.25,0) -- (-5.25,0) node[above] {free space};
   \draw[shift={(6.5,0)},fill=black] (0,0) circle (1pt) node[below] {$b_1 = b$};
   \draw[shift={(3,0)},fill=black] (0,0) circle (1pt) node[below] {$a_1$};
   \draw[shift={(2.5,0)},fill=black] (0,0) circle (1pt) node[below] {$b_2$};
   \draw[shift={(-.5,0)},fill=black] (0,0) circle (1pt) node[below] {$a_2$};
   \draw[shift={(-1,0)},fill=black] (0,0) circle (1pt) node[below] {$b_r$};
   \draw[shift={(-2.5,0)},fill=black] (0,0) circle (1pt) node[below] {$a_r$};
   \draw[shift={(-7.5,0)},fill=black] (0,0) circle (1pt) node[below] {$a$};
        \foreach \x in {-15,...,13}{
  \draw[shift={(\x/2,0)},color=black] (0pt,3pt) -- (0pt,-3pt);

}
\end{tikzpicture}
\end{center}

A \emph{valid sequence} for an interval $[a,b]$ is a sequence with $1+b-a$ integers,
whose first term is $b$ and in which every term of the interval appears exactly once.  If $S:\,s_1,s_2,\ldots,s_m$ is any sequence of positive integers, 
we denote by $(S)$ the permutation $(s_1,s_2,\ldots,s_m)$.  We say two valid sequences $S$ and $\bar S$ {\it for $[a,b]$} (i.e. each of $S$ and 
$S'$ contains exactly the integers in $[a,b]$) are \emph{opposite}
if $(S)$ and $(\bar S)$ are  conjugate by an odd permutation in $\Perm([a,b])$.
Given a packing $I_1,\ldots,I_r$ of $[a,b]$ and a valid sequence $S_i$ for each  $I_i$ (i.e. $S_i$ contains precisely the integers in $I_i$), 
we define the \emph{packing cycle} (for these valid sequences) to be the $\len [a,b]$-cycle obtained by juxtaposing 
$S_1,\ldots,S_r$ in that order, following this sequence by the terms $a_r-1,a_r-2,\ldots,a$ of the free space in decreasing order, and applying $()$.
This construction is motivated by the following fact.

\begin{lem}
\label{Packing cycle}
Suppose $I_1,\ldots,I_r$ is a packing of $[a,b]$ with valid sequences $S_1,\ldots,S_r$, respectively, and $\delta$ is the packing cycle. 
For $1\le i\le r$, let
$$\beta_i  := (a_i,a_i+1,\ldots,b_i)(S_i)$$
Then the $\beta_i$ are disjoint permutations, and
$$(a,a+1,\ldots,b)\delta= \prod_{i=1}^r \beta_i.$$
\end{lem}

\begin{proof}
Each $\beta_i$ is a permutation in the elements of $[a_i,b_i]$, which are disjoint.  To prove the equality, we
compare how each side acts on $x\in [a,b]$.  If $x$ is the $j^{\mathrm{th}}$ term $s_{i,j}$ of the sequence $S_i$ but not the last term in the sequence, then
both sides map $x\mapsto s_{i,j+1}+1$.  If $x$ is the last term in $S_i$ then both sides map $x\mapsto a_i$ (though for the left hand side for $i=r$, one
has to check separately the cases $a_r>a$ and $a_r=a$.)
If $a\le x\le a_r$, then both sides map
$x$ to itself (though for the left hand side, we have to check the cases $x>a$ and $x=a$ separately.)
\end{proof}

For any given permutations $\gamma,\delta$, the following lemma allows us to modify $\delta$ so that $\gamma\delta$ can have longer cycles.

\begin{lem}
\label{Rebuild long cycles}
Let $\gamma,\delta\in \Perm(X)$ for some finite set $X$.  Suppose $x,y$ are distinct elements of $X$ with the following properties
\begin{enumerate}[\rm(a)]
\item $\delta$ fixes neither $x$ nor $y$,
\item $\gamma(x)$ belongs to $\cO_{\gamma\delta}(x)$, the $\gamma\delta$-orbit of $x$, and
\item $y$ and $\gamma(y)$ are both fixed by $\gamma\delta$.
\end{enumerate}
Then there exists $\delta'\in \Perm(X)$, conjugate to $\delta$ by a transposition, such that $\gamma\delta'$ coincides with
$\gamma\delta$ except on $\cO_{\gamma\delta}(x)\cup\{y,\gamma(y)\}$, which forms a single orbit of size $2+|\cO_{\gamma\delta}(x)|$ 
under $\gamma\delta'$.
\end{lem}

\begin{proof}
By (a) and (b), $\cO_{\gamma\delta}(x)$ cannot consist just of $x$, so no element in this orbit is fixed by $\gamma\delta$.
This and (c) imply that neither $y$ nor $\gamma(y)$ can belong to $\cO_{\gamma\delta}(x)$.
If $\gamma$ fixed $y$, then (c) would imply $\gamma\delta(y) = y=\gamma(y)$, contrary to (a). We have shown that
\begin{equation}\label{gd1}
  |\cO_{\gamma\delta}(x)\cup\{y,\gamma(y)\}| = 2+|\cO_{\gamma\delta}(x)|.
\end{equation}  

We define $\delta' := \epsilon\delta\epsilon$, where $\epsilon := (x,y)$, so $\delta'$ is conjugate to $\delta$ by a transposition.  
If 
$$z\not\in \{x,y,\delta^{-1}(x),\delta^{-1}(y)\},$$
then
$$\gamma\delta'(z) = \gamma\epsilon\delta\epsilon(z) = \gamma\epsilon\delta(z) = \gamma\delta(z).$$
By (b), 
$$\delta^{-1}(x) = (\gamma\delta)^{-1}(\gamma(x))\in \cO_{\gamma\delta}(x),$$
and by (c),
\begin{equation}\label{gd2}
  \delta^{-1}(y) = (\gamma\delta)^{-1}(\gamma(y)) = \gamma(y).
\end{equation}  
Therefore, outside $\cO_{\gamma\delta}(x)\cup\{y,\gamma(y)\}$, replacing $\gamma\delta$ by $\gamma\delta'$ has no effect.

On the other hand, note that $\delta(y) \neq y$ by (a), and $\delta(y) \neq x$ (as otherwise we would have 
$\gamma\delta(y) = \gamma(x) \in \cO_{\gamma\delta}(x)$ by (b) and so $y \in \cO_{\gamma\delta}(x)$, contrary to \eqref{gd1}).
Next, $\delta(x) \neq x$ by (a), and $\delta(x) \neq y$ (as otherwise we would have 
$\gamma(y) = \gamma\delta(x) \in \cO_{\gamma\delta}(x)$, again contradicting \eqref{gd1}). It then follows that
$\delta^{-1}(x) \neq x,y$. We also note from \eqref{gd1} that $\gamma(y) \neq x,y$. Furthermore,
$\delta\gamma(y)=y$ by \eqref{gd2}. Now we can check that
\begin{align*}
\gamma\delta'(x) &= \gamma\epsilon\delta\epsilon(x) = \gamma\epsilon\delta(y) = \gamma\delta(y) = y,\\
\gamma\delta'(y) &= \gamma\epsilon\delta\epsilon(y) = \gamma\epsilon\delta(x) = \gamma\delta(x),\\
\gamma\delta'\delta^{-1}(x) &= \gamma\epsilon\delta\epsilon\delta^{-1}(x) =  \gamma\epsilon\delta\delta^{-1}(x) = \gamma\epsilon(x) = \gamma(y),\\
\gamma\delta'\gamma(y) &= \gamma\epsilon\delta\epsilon\gamma(y) = \gamma\epsilon\delta\gamma(y)  = \gamma\epsilon(y) = \gamma(x).
\end{align*}
Thus, on $\cO_{\gamma\delta}(x) \cup \{y,\gamma(y)\}$ we replace
$$
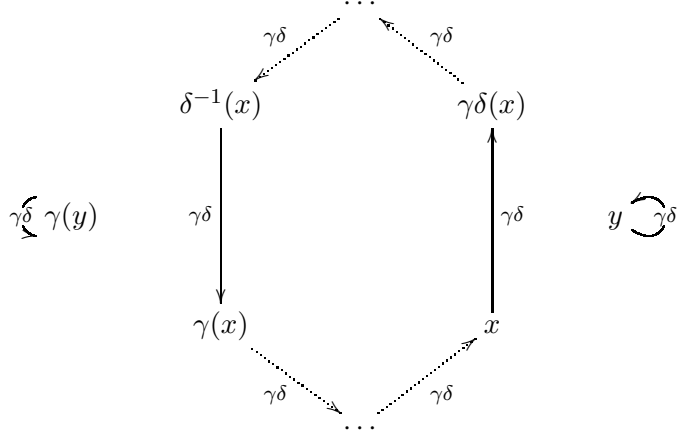
\xymatrix{&&\cdots\ar@{.>}[dl]_{\gamma\delta}\\
&\delta^{-1}(x)\ar[dd]_{\gamma\delta}&&\gamma\delta(x)\ar@{.>}[ul]_{\gamma\delta} \\
\gamma(y) \ar@(ul,dl)[]|{\gamma\delta} &&&& y \ar@(dr,ur)[]|{\gamma\delta} \\
&\gamma(x)\ar@{.>}[dr]_{\gamma\delta}&&x\ar[uu]_{\gamma\delta}\\
&&\cdots\ar@{.>}[ur]_{\gamma\delta}}$$
by
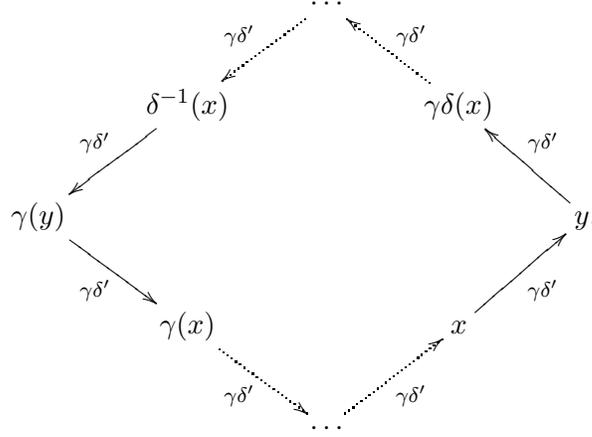
$$\xymatrix{&&\cdots\ar@{.>}[dl]_{\gamma\delta'}\\
&\delta^{-1}(x)\ar[dl]_{\gamma\delta'}&&\gamma\delta(x)\ar@{.>}[ul]_{\gamma\delta'} \\
\gamma(y) \ar[dr]_{\gamma\delta'} &&&& y. \ar[ul]_{\gamma\delta'} \\
&\gamma(x)\ar@{.>}[dr]_{\gamma\delta'}&&x\ar[ur]_{\gamma\delta'}\\
&&\cdots\ar@{.>}[ur]_{\gamma\delta'}}$$
\end{proof}

We say that $x$ is {\it special for $(\gamma,\delta)$} if property \ref{Rebuild long cycles}(b) and the first part of property \ref{Rebuild long cycles}(a) 
hold for $x$.
If $x$ is special for $(\gamma,\delta)$, then it is special for $(\gamma,\delta')$ with $\delta'$ constructed in Lemma 
\ref{Rebuild long cycles}.  Therefore, we can iterate this process $N$ times 
if we have $N$ mutually disjoint pairs $\{y_i,\gamma(y_i)\}$ such that 
$$\delta(y_i)\neq y_i,~\gamma\delta(y_i) = y_i, \mbox{ and } \gamma\delta\gamma(y_i) = \gamma(y_i)$$ 
in $X \smallsetminus \cO_{\gamma\delta}(x)$.

\smallskip
We now prove Theorem~\ref{Construction}.

\begin{proof}
Suppose $g\in \AAA_n$ has $\fix(g) \geq 8k+9$ fixed points.  Then $g^{\AAA_n} = g^{\SSS_n}$. Since the case $\lambda$ corresponds to a 
unique $\AAA_n$-class is already covered by \cite[Proposition~6.1]{LS2}, it suffices to prove
that there are $\SSS_n$-conjugates of $g$ in $C^2$ and $CD$, where $C$ and $D$ are the two distinct $\AAA_n$-conjugacy classes associated to $\lambda$.

\smallskip
For $j=1,2,\ldots,k$, let $\Interval_j := [1,\lambda_j]$.    
We choose embeddings $\iota_1,\ldots,\iota_k$ of the $\Interval_i$ in $\{1,2,\ldots,n\}$
with disjoint images, and we identify the Young subgroup
$$\SSS_\lambda=\Perm(\Interval_1)\times \cdots\times \Perm(\Interval_k)$$ 
with a subgroup of $\SSS_n$ via the $\iota_j$.  Let 
\begin{equation}\label{gamma}
  \gamma\in\SSS_\lambda < \SSS_n
\end{equation}  
denote the element whose $j^{\mathrm{th}}$ factor is the $\lambda_j$-cycle $(1,2,\ldots,\lambda_j) \in \Perm(\Interval_j)$
for $1\le j\le k$. Without loss we may assume that $\gamma \in C$.

\smallskip
Next we consider the partition 
$$\mu = (\mu_1, \ldots,\mu_l)$$ 
associated to (the cycle type of) $g$.  We can break it into subpartitions, each of one of the following kinds:
\begin{enumerate}
\item one part of size $1$
\item four parts of size $1$ and one part of size $3$
\item three parts of size $3$
\item one odd part of size $m\ge 5$
\item two parts of size $2$
\item four parts of size $2$
\item one part of size $2$ and one even part of size $m\ge 4$
\item one part each of even sizes $m_1\ge m_2\ge 4$ (or two parts of size $m_1$ if $m_1=m_2$).
\end{enumerate}
(Here, if $\nu= (\nu_1, \ldots,\nu_s)$, then we say the parts of $\nu$ have size $\nu_1, \ldots,\nu_s$.)

We never need more than two subpartitions of type (2) or more than one of type (5).  In particular, since parts of size $1$ only appear in subpartitions of type (1) and (2), the number of parts of type (1) must be at least $\fix(g)-8\ge 8k+1$.

\smallskip
We define a function $\Phi$ on subpartitions which sends partitions of type (1), (2), (3), (4), (5), (6), (7), and (8) respectively into the empty partition
$\varnothing$, $1^4 3^1$, $3^3$, $5^1$, $2^2$, $2^4$, $2^1 4^1$, and $4^2$ respectively.
Except for type (1), the effect of $\Phi$ is to take every part which is $6$ or larger and replace it by an integer of the same parity in the set $\{4,5\}$; otherwise, it fixes all parts.  
In particular, it fixes partitions of types (2), (3), (5), and (6).  If we break a partition of $n$ into subpartitions, apply $\Phi$ to each, recombine, and add parts of size $1$ as needed to obtain a new partition of $n$, the new partition is obtained
from the old by replacing each odd part $m\ge 7$ by $1^{m-5}5^1$ and each even part $m\ge 6$ by $1^{m-4}4^1$. This function $\Phi$ will allow
us to modify products of permutations from $C^2$, respectively $CD$, iterating the construction in Lemma \ref{Rebuild long cycles}.

\smallskip
We next choose packings of each $\Interval_j$ using the greedy algorithm, that is, packing the largest subintervals first, and we assign
each subinterval in each $\Interval_j$ to one of the subpartitions of $\mu$ not of type (1) so that each subpartition is
associated to exactly one $\Interval_j$ and exactly one subinterval.  Moreover, the subintervals associated to subpartitions of type (2) have length 7; for type (3), length 9; for type (4), length 5; for type (5), length 4; 
for type (6), length 8; for type (7), length 6; and for type (8), length 8.  Note that the subinterval is never longer than the size of the subpartition.
Such an assignment is always possible because the free space left after packing an interval, if there is any, is always an interval with the same left endpoint.
As long as the free interval is of length $\ge 9$, there is room to pack another subinterval associated to one of the subpartitions.  Therefore, as long as at least $8k+1$ total integers
remain, there will be room for another subinterval associated to a subpartition.  This is guaranteed because there are at least $8k+1$ subpartitions of type (1), and they are not assigned subintervals.
When we are finished with the packings, let $X_j$ denote the free initial subinterval, which is the complement in $\Interval_j$ of the union of all the subintervals $I_i$ in its packing.

Next, for each subinterval $I_i$, we choose a pair of opposite valid sequences, $S_i$ and $\bar S_i$ except that for type (5) we choose only a single sequence $S_i$.
The sequences have the property that $(1,2,\ldots,\len(S_i))(S_i)$ and $(1,2,\ldots,\len(S_i))(S'_i)$ (except in type (5)) have the shape given by applying $\Phi$ to the corresponding subpartition.
For lengths $7$ and $9$, the fact that this can be done follows from Theorem~\ref{Gleason}.  For length $5$ it follows from the fact that $5$-cycles are covered by $5^1$.
In the even length cases, it is easy to find sequences which work: 
for type (5), $(4,1,2,3)$; 
for type (6), $(8,1,2,7,4,5,6,3)$ and $(8,1,3,5,6,2,7,4)$; 
for type (7), $(6,1,2,4,5,3)$ and $(6,1,3,4,5,2)$; 
and for type (8), $(8,1,2,4,7,5,3,6)$ and $(8,1,2,5,7,3,6,4)$.
Note that every orbit of length $4$ or $5$ in $(1,2,\ldots,\len(S_i))(S_i)$ has at least one special element; this is obvious for type (4) and can be easily checked for types (7) and (8).

\smallskip
For each interval $\Interval_j$, consider its packing $I_1,\ldots,I_r$.  Choose for each $I_i$ either $S_i$ or $\bar S_i$ (if the latter exists), and define $\delta_j$ to be the
packing cycle for these sequences. 
For instance, if $\lambda_j = 15$, $r=2$, $I_1$ is of type (6), $I_2$ is of type (5), and we choose $\bar S_1$ rather than $S_1$, then we obtain the following sequence:
$$\delta_j = (\underbrace{15, 8, 10, 12, 13, 9, 14, 11}_{\bar S_1}, \underbrace{7, 4, 5, 6}_{S_2}, \underbrace{3, 2, 1}_{X_j}).$$
We define $\delta\in \AAA_n$ to be 
\begin{equation}\label{delta}
  \delta = (\delta_1,\ldots,\delta_k)\in \SSS_{\lambda_1}\times \cdots\times\SSS_{\lambda_k}
\end{equation}  
regarded as an element of $\SSS_n$ associated to $\lambda$.
As long as there is at least one subpartition of $\mu$ not of type (1) or type (5), there are pairs of choices $\delta$ and $\bar\delta$ which are not conjugate to one another in $\AAA_n$. In particular, one of $\gamma\delta, \gamma\bar\delta$ belongs to $C^2$ and the other belongs to $CD$.
By Lemma~\ref{Packing cycle}, the shape of the products $\gamma\delta$ and $\gamma\bar\delta$ are the same, obtained from the shape of $g$
by taking each odd part $\mu_i\ge 7$ and replacing it with $1^{\mu_i-5} 5^1$
and taking each even part $\mu_i\ge 6$ and replacing it with the shape $1^{\mu_i-4}4^1$.

\smallskip
For each interval $\Interval_j$ which has a subinterval $[1,F_j]$ of free space, we define $Y_j$ to be the set of {\it even} elements in $[1,F_j]$
strictly less than $F_j$, so $Y_j$ is the larger of $0$ and $\lfloor (F_j-1)/2\rfloor$.  If $\Interval_j$ has no free space, $Y_j$ is empty.
We define $Y$ to be the union of $\iota_j(Y_j)$ over $1\le j\le r$.  
Let $X$ be the union of $\iota([0,F_j])$ over all $j$ for which $\Interval_j$ has free space.  Thus $X$ is the set of all free space, and
\begin{equation}
\label{xy}
\frac{|X|}2 = \sum_{j=1}^k \frac{F_j}2\le \sum_{j=1}^k \frac{2|Y_j|+2}2 = |Y|+k.
\end{equation}

For each $y \in Y$, $\gamma(y) \neq y$ by the choice of $\gamma$. 
Next, the set of pairs $\{y,\gamma(y)\}$ as $y$ ranges over $Y$ are mutually disjoint and disjoint from all $\gamma\delta$-orbits of length $\ge 2$.
Moreover, $\delta(y) \neq y$ and $\gamma\delta$ fixes $y$ and $\gamma(y)$ for all $y\in Y$.  By Lemma~\ref{Rebuild long cycles}, we may repeat up to $|Y|$ times 
the modification process of taking any $\gamma\delta$-orbit which 
contains a special point and increase its length by $2$ while reducing the number of fixed points of $\gamma\delta$ by $2$.  We need to repeat it $\max(0,\lfloor \pi_i/2\rfloor-2)$
times for each part $\pi_i$ of $\pi$, so the number of repetitions does not depend on whether we choose $\delta$ or $\bar \delta$, and the conjugates of $\delta$ and $\bar\delta$ resulting from this iterated modification process
will be conjugate by an odd permutation of $\SSS_n$. Thus, at each step of this process, we obtain a pair of two elements, one belonging to
$C^2$ and the other to $CD$.

\smallskip
Now we check that $Y$ is large enough to reverse the process which goes from the shape of $g$ to the shape of $\gamma\delta$,
that is,
$$|Y| \ge \sum_i \max(0,\lfloor \mu_i/2\rfloor-2).$$
The amount of free space for $g$ is
$$|X| = |\{i\mid \mu_i=1\}| +\sum_i \max(2\lfloor \mu_i/2\rfloor-4,0),$$
so by \eqref{xy},
\begin{align*}
|Y| & \ge \frac{|\{i\mid \mu_i=1\}| +\sum_i \max(2\lfloor \mu_i/2\rfloor-4,0)}2 - k\\ 
&> 3k + \frac{\sum_i \max(2\lfloor \mu_i/2\rfloor-4,0)}2 \\
& > \sum_i \max(0,\lfloor \mu_i/2\rfloor-2).
\end{align*}
The fact that we can go back from the shape of $g$ to the shape of $\gamma\delta$, with $\gamma$ given in \eqref{gamma} and 
$\delta$ given in \eqref{delta}, shows that at the end of the iterated modification process, 
we obtain one element from $C^2$ and another element from $CD$, both $\SSS_n$-conjugate to $g$.

\smallskip
This finishes the proof as long as at least one of the subpartitions of $\mu$ is not of type (1) or (5).  Since we use at most
one subpartition of type (5), the only cases still to be considered are the case that $g$ is of type $1^n$ (which is excluded by
hypothesis) and the case that $g$ is of type $1^{n-4}2^2$.  If $n>9$, we have
$\lambda_1\ge 7$, so this last case follows from Lemma~\ref{Cancel} and Theorem~\ref{Gleason}.
\end{proof}

\section{General classes}

In this section we consider products $x^{\AAA_n} y^{\AAA_n}$ where $x,y\in\AAA_n$ are elements whose conjugacy classes in $\AAA_n$ are strictly contained in their classes in $\SSS_n$.

\begin{lem}
\label{AMGM}
For all $\epsilon > 0$, there exists $N=N_1(\epsilon) > 0$ such that if $n\ge N$ and
$a_1+a_2+\cdots+a_m = n$, where the $a_i$ are pairwise distinct positive integers, then $a_1\cdots a_m < e^{\epsilon n}$.
\end{lem}

\begin{proof}
By the arithmetic-geometric mean inequality, 
$$a_1\cdots a_m \le (n/m)^m = \bigl((n/m)^{m/n}\bigr)^n.$$
Since the sum $n=\sum^m_{i=1}a_i$ of $m$ distinct positive integers is at least $m(m+1)/2$, we have $m<\sqrt{2n}$, so $n/m > \sqrt{n/2}$.  As
$$\lim_{x\to\infty} \frac{\log x}x = 0,$$
for $n$ sufficiently large, we have
$$\frac{\log n/m}{n/m} < \epsilon,$$
and so
$$\bigl((n/m)^{m/n}\bigr)^n < \bigl(e^\epsilon\bigr)^n = e^{\epsilon n}.$$
\end{proof}

\begin{prop}
\label{Decomposable}
For all $\epsilon > 0$, there exists $N=N_2(\epsilon) > 0$ such that if $n\ge N$ and
$\chi$ is any irreducible character of $\AAA_n$ which is not the restriction of any irreducible character of $\SSS_n$, then $\chi(1) \ge e^{(\log 2-\epsilon)n}$.
\end{prop}

\begin{proof}
Let $\lambda$ be the partition of $n$ for which $\chi$ is an irreducible constituent of the restriction from $\SSS_n$ to $\AAA_n$ of $\chi_\lambda$.
By \cite[Theorem~2.5.7]{JK}, the Young diagram of $\lambda$ must be equal to its transpose.  Let $m$ be the largest integer such that $\lambda_m\ge m$, and let $a_i:= \lambda_i$ for $1\le i\le m$,
so $(a_1,\ldots,a_m|a_1,\ldots,a_m)$ is the Frobenius symbol of $\lambda$.  By \cite[Theorem~2.2]{LS1}, for $N$ sufficiently large, we have
$$1-\epsilon < \frac{\log (n-1)! - 2\sum_{i=1}^m \log a_i!}{\log \chi_\lambda(1)} < 1+\epsilon.$$
As $a_1!\cdots a_m!$ divides $(a_1+\cdots+a_m)!$, which in turn divides $\lfloor \frac{n-1}2 \rfloor !$, we have
$$\frac{\log\binom{n-1}{\lfloor (n-1)/2 \rfloor}}{\log \chi_\lambda(1)} < 1+\epsilon,$$
so
$$\chi(1) = \frac{\chi_\lambda(1)}2 > \frac{\binom{n-1}{\lfloor (n-1)/2 \rfloor}^{1-\epsilon}}2 > \frac{(2^n)^{1-\epsilon}}{4n} > e^{(\log 2-\epsilon)n}$$
when $n$ is sufficiently large.
\end{proof}

\begin{prop}
\label{few small parts}
For all $M \geq 1$ there exists $N=N_3(M)$ such that if $n\ge N$ and $x\in \AAA_n$, regarded as an element of $\SSS_n$, has at most $M$ cycles of length $\le M$, then for all irreducible characters $\chi$ of $\AAA_n$,
$$|\chi(x)| \le \chi(1)^{1/M}.$$
\end{prop}

\begin{proof}
Given an element $x\in \SSS_n$,
following \cite{LS1}, we define the  sequence $e_1,\ldots,e_n$ of non-negative real numbers associated to $x$ so that,
for each $1 \leq k \leq n$, $n^{e_1+\cdots+e_k}$
equals the number of elements in $\{1,2,\ldots,n\}$ whose $x$-orbit has length $\le k$, assuming there is at least one.  (If every $x$-orbit has length $>k$, then we set $e_1=\cdots=e_k=0$).
We define $E$ (or $E(x)$ when we want to make the choice of $x$ explicit) to be $e_1 + e_2/2+\dots+e_n/n$.  Note that $e_1+\cdots+e_n=1$.  
If there are at most $M$ cycles of length $\le M$, then $e_1+\cdots+e_M \le \log_n M^2$.  Therefore,
\begin{equation}
\label{E-ineq}
E = \sum_{i=1}^M e_i/i + \sum_{i=M+1}^N e_i/i \le \log_n M^2 + \frac 1{M+1}.
\end{equation}

Then \cite[Theorem~1.1]{LS1} asserts that for all $\epsilon >0$ there exists $N=N_4(\epsilon)$ such that $n>N$ implies
$$|\chi(x)| \le \chi(1)^{E+\epsilon}$$
for all irreducible characters $\chi$ of $\SSS_n$.  Fixing $\epsilon < 1/M - 1/(M+1)$, if $N$ is large enough in terms of $M$ and $n\ge N$, then by \eqref{E-ineq}, $E < 1/M - \epsilon$.  The proposition follows
for all irreducible characters of $\AAA_n$ which extend to characters of $\SSS_n$.

We may therefore assume that the Young diagram of $\lambda\vdash n$ equals its transpose, and $\chi$ is one of the two irreducible constituents obtained by restricting $\chi_\lambda$ from $\SSS_n$ to $\AAA_n$.
By \cite[Theorem~2.5.13]{JK}, for each $x$, either $\chi(x) = \chi_\lambda(x)/2$ or  
$$|\chi(x)| \le \frac{1+\prod_{i=1}^m \sqrt{a_i}}2,$$
where the integers $a_i$ are distinct and odd and have sum $n$.  In the first case, the proposition follows from \cite[Theorem~1.1]{LS1}.
In the second case, it follows from Lemma~\ref{AMGM} and Proposition~\ref{Decomposable}.
\end{proof}

Theorem~\ref{char-bound} is now an immediate consequence.

\begin{proof}
If the conjugacy class of $x\in \AAA_n$ is properly contained in its conjugacy class in $\SSS_n$, then the partition associated to $x$ as an element of $\SSS_n$ consists of distinct odd parts. Choosing $M = 1/\epsilon$, we see that $x$ has at most $M$ cycles of 
length $\leq M$, and, therefore, the statement follows from Proposition~\ref{few small parts} applied to $x$.
\end{proof}

We can now prove Theorem~\ref{Covering}.

\begin{proof}
By Theorem~\ref{char-bound}, if $n$ is sufficiently large, every element $x\in \AAA_n$ 
whose associated partition $\mu$ has distinct odd parts
satisfies
$$|\chi(x)| \le \chi(1)^{1/18}$$
for all irreducible characters $\chi$ of $\AAA_n$.  
The same bound of course holds for any element $y \in \AAA_n$ 
whose associated partition is $\mu$.
By \cite[Theorem~1.1]{LS1}, if $n$ is sufficiently large and $z\in \AAA_m$ has at most $n^{3/5}$ fixed points, we have
$$|\chi(z)| \le \chi(1)^{5/6}$$
for all irreducible $\chi$.  Therefore,
$$\Bigm|\sum_\chi \frac{\chi(x)\chi(y)\chi(z)}{\chi(1)}\Bigm| \ge 1 - \sum_{\chi\neq 1_{\AAA_m}} \frac{|\chi(x)\chi(y)\chi(z)|}{\chi(1)} \ge 1 - \sum_{\chi\neq 1_{\AAA_n}} \chi(1)^{-1/18}.$$
By \cite[Corollary~2.7]{LiS}, the right hand side is positive if $n$ is sufficiently large.  It follows that every element of $\AAA_n$ with at most $n^{3/5}$ fixed points is covered by every
partition $\mu$ of $n$ with distinct odd parts.

On the other hand, a partition of $n$ with distinct odd parts has at most $\sqrt n$ such parts; in particular this holds for $\lambda$. So if $n$ is sufficiently large, the hypotheses of Theorem~\ref{Construction} are satisfied for $\lambda$ and every element with
at least $n^{3/5}$ fixed points.  The theorem follows.
\end{proof}

Next we prove Theorem~\ref{Normal Subset}.

\begin{proof}
If $W$ contains any element whose $\AAA_n$-conjugacy class is a proper subset of its $\SSS_n$-conjugacy class, then $W^2 \supseteq \AAA_n\setminus \{1\}$ by Theorem~\ref{Covering}.
We may therefore assume it is a union of $\SSS_n$-conjugacy classes.  

Fix $\epsilon \in (0,1/4-\alpha)$ and $\beta\in (\alpha,1/4-\epsilon)$.  By \cite[Corollary~6.5]{LS1}, there exists $N$ such that for all $n>N$, the number of elements $x\in \SSS_n$ satisfying
$E(x) \ge1/4-\epsilon$ is less than $e^{-n^\beta} n!$, which we may assume is less than $e^{-n^\alpha} |\AAA_n|$.  Therefore, we may assume $W$ contains a subset of the form $x^{\SSS_n}$
where $E(x) < 1/4-\epsilon$.  By \cite[Corollary~1.11]{LS1}, this implies that $W^2$ contains $\AAA_n$.
%
%
\end{proof}

We conclude with a proof of Theorem \ref{an-cn}.

\begin{proof}
The case $n=5$ of (i) can be checked using \cite{GAP}.
For (i) with $n \geq 7$ and for (ii), Theorem \ref{Gleason}, respectively Theorem \ref{Covering}, shows that $C^2 \supseteq \AAA_n \smallsetminus \{1\}$.
Given any $x \in \AAA_n$, we can find $y \in C$ such that $y \neq x$. Hence $xy^{-1} \in C^2$, and so $x=(xy^{-1})y \in C^3$, showing 
$C^3 = \AAA_n$ and thus $\cn(C) \leq 3$. Clearly $\cn(C) > 1$. We have therefore shown that $\cn(C) = 2$ if and only if 
$1 \in C^2$, i.e. $g$ is a real element in $\AAA_n$. By assumption, $g^{\AAA_n} \neq g^{\SSS_n}$, so $g$ is a disjoint product of cycles
of distinct odd length. To complete the proof, it now suffices to note that, given any $m$-cycle $h \in \AAA_m$ (so $2 \nmid m$), any 
permutation in $\SSS_m$ that conjugates $h$ to $h^{-1}$ is even if $m \equiv 1 \pmod{4}$, and odd if $m \equiv 3 \pmod{4}$.  
\end{proof}


\begin{thebibliography}{LaST3}

\bibitem{Bertram} 
E.~Bertram, 
Even permutations as a product
of two conjugate cycles. \textit{J. Comb. Th. Ser. A} \textbf{12} (1972), 
368--380.

\bibitem{Biane}
P.~Biane, 
Representations of symmetric groups and free probability.
{\it Adv.\ Math.} \textbf{138} (1998), no.\ 1, 126--181. 

\bibitem{Brenner} 
J.~L.~Brenner, 
Covering theorems for finite
nonabelian simple groups.  IX.  How the square of a class
with two nontrivial orbits in $S_n$ covers $A_n$,
\textit{Ars Combinatorica} \textbf{4} (1977), 151--176.

\bibitem{BR} 
J.~L. Brenner and J.~Riddell, 
Covering theorems for finite nonabelian simple groups. VII. 
Asymptotics in the alternating groups, 
\textit{Ars Combinatorica} \textbf{1} (1976), 77--108.


\bibitem{FL}
  S.~Fomin and N.~Lulov, 
  On the number of rim hook tableaux.
  \textit{J.\ Math.\ Sci.\ (New York)} \textbf{87} (1997), no.\ 6, 4118--4123.
  
\bibitem{GAP}
The GAP group, {\sf GAP} -- groups, algorithms, and
programming. Version 4.4, 
2004,
\url{http://www.gap-system.org}.  

\bibitem{Hsu}
C.~Hsu,
The commutators of the alternating group. 
\textit{Sci.\ Sinica} \textbf{14} (1965), 339--342. 

\bibitem{Husemoller}
D.~Husemoller, 
Ramified coverings of Riemann surfaces.
\textit{Duke Math. J.} \textbf{29} (1962), 167--174. 

\bibitem{JK}
G.~James and A.~Kerber, 
The representation theory of the symmetric group. 
With a foreword by P.~M.~Cohn. With an introduction by Gilbert de B.\ Robinson. 
Encyclopedia of Mathematics and its Applications, 16. 
Addison-Wesley Publishing Co., Reading, Mass., 1981.

\bibitem{LS1}
  M.~Larsen and A.~Shalev,
  Characters of symmetric groups: sharp bounds and applications. 
  \textit{Invent. Math.} \textbf{174} (2008), no.\ 3, 645--687.

\bibitem{LS2}
  M.~Larsen and A.~Shalev,
  Word maps and Waring type problems. \textit{J.\ Amer.\ Math.\ Soc.} \textbf{22} (2009), no.\ 2, 437--466.

\bibitem{LST0}
M.~Larsen, A.~Shalev, and Pham Huu Tiep,
Products of normal subsets and derangements,
arXiv: 2003.12882.

\bibitem{LST}
M.~Larsen, A.~Shalev, and Pham Huu Tiep,
Products of normal subsets, \textit{Trans.\ Amer.\ Math.\ Soc.} (to appear).
 \bibitem{LiS}
 
 M.~W.~ Liebeck, and A.~Shalev,  
 Fuchsian groups, coverings of Riemann surfaces, subgroup growth, random quotients and random walks. 
 \textit{J.\ Algebra} \textbf{276} (2004), no.\ 2, 552--601. 
  
\bibitem{MSP}
  T.~W.~M\"uller and J-C.~Schlage-Puchta,  
  Character theory of symmetric groups, subgroup growth of Fuchsian groups, and random walks. 
 \textit{Adv.\ Math.} \textbf{213} (2007), no.\ 2, 919--982.
  
\bibitem{RS}
  A. Rattan and P. \'Sniady, 
  Upper bound on the characters of the symmetric groups for balanced
  Young diagrams and a generalized Frobenius formula.
  \textit{ Adv.\ Math.} \textbf{218} (2008), 673--695.

\bibitem{Roichman}
  Y.~Roichman, Upper bound on the characters of the symmetric groups.
  \textit{Invent.\ Math.} \textbf{125} (1996), 451--485.

  
\end{thebibliography}
\end{document}